\documentclass[10pt,reqno]{amsart}

\usepackage{dsfont}
\usepackage[bottom]{footmisc}
\usepackage{amssymb, amsmath, amsthm}
\usepackage{enumerate,color}
\newtheorem{thm}{Theorem}[section]
\newtheorem{lem}[thm]{Lemma}

\newtheorem{defn}[thm]{Definition}

\numberwithin{equation}{section}

\newcommand{\bel}{\begin{equation} \label}
\newcommand{\ee}{\end{equation}}
\def\beq{\begin{equation}}
\def\eeq{\end{equation}}
\newcommand{\bea}{\begin{eqnarray}}
\newcommand{\eea}{\end{eqnarray}}
\newcommand{\beas}{\begin{eqnarray*}}
\newcommand{\eeas}{\end{eqnarray*}}

\newcommand{\wh}{\widehat}

\newcommand{\re}{\mathfrak R}

\newcommand{\R}{\mathbb{R}}
 
\newcommand{\N}{\mathbb{N}}

\newcommand{\cA}{\mathcal{A}}
\newcommand{\cB}{\mathcal{B}}

\newcommand{\cL}{\mathcal{L}}

\newcommand{\cS}{\mathcal{S}}
\newcommand{\cR}{\mathcal{R}}

\newcommand{\supp}{\mathrm{supp}\,}  %supp

\allowdisplaybreaks
\def\epsilon{\varepsilon}
\def\phi {\varphi}

\providecommand{\abs}[1]{\left\lvert#1\right\rvert}
\providecommand{\norm}[1]{\left\lVert#1\right\rVert}

\renewcommand{\leq}{\leqslant}
\renewcommand{\geq}{\geqslant}

\providecommand{\abs}[1]{\left\lvert#1\right\rvert}
\providecommand{\norm}[1]{\left\lVert#1\right\rVert}

\title[Solving singular time-fractional diffusion equations]
{\bf Solving time-fractional diffusion equations with singular source term}
\author{Yavar Kian and Eric Soccorsi}
\address{Aix Marseille Univ, Universit\'e de Toulon, CNRS, CPT, Marseille, France}

\email{yavar.kian@univ-amu.fr}

\address{Aix Marseille Univ, Universit\'e de Toulon, CNRS, CPT, Marseille, France}
\email{eric.soccorsi@univ-amu.fr}

\date{}

\begin{document}

\begin{abstract} This article deals with time-fractional diffusion equations with time-dependent singular source term. Whenever the order of the time-fractional derivative is either multi-term, distributed or space-dependent, we prove that the system admits a unique weak solution enjoying a Duhamel representation, provided that the time-dependence of the source term is a distribution. 
\end{abstract}

\maketitle

%%%%%%%%%%%%%%%%%%%%%%%%%%%%%%%%%%%%%%%%%%%%%%%%%%%%%%%%%%%%%%%%%
%%%%%%%%%%%%%%%%%%%%%%%%%%%%%%%%%%%%%%%%%%%%%%%%%%%%%%%%%%%%%%%%%
%%%%%%%%%%%%%%%%%                                    Introduction                             %%%%%%%%%%%%%%%%%%%%%%
%%%%%%%%%%%%%%%%%%%%%%%%%%%%%%%%%%%%%%%%%%%%%%%%%%%%%%%%%%%%%%%%%
%%%%%%%%%%%%%%%%%%%%%%%%%%%%%%%%%%%%%%%%%%%%%%%%%%%%%%%%%%%%%%%%%

\section{Introduction and settings}

\subsection{Time-fractional derivatives}
\label{sec-settings}

In the present article, $\Omega$ is a bounded and connected open subset of $\R^d$, $d \geq 2$, 
with Lipschitz  boundary $\partial \Omega$. Given $K\in L^1_{\mathrm{loc}}(\R_+,L^\infty(\Omega)) \cap C^\infty(\R_+,L^\infty(\Omega))$, where $\R_+:=(0,+\infty)$, we introduce the integral operator
$$(I_Kg)(t,x) :=\int_0^t K(t-s,x)g(s,x)ds,\ g\in L^1_{\mathrm{loc}}(\R_+,L^2(\Omega)),\ x\in\Omega,\ t\in\R_+.$$
For any complete locally convex topological vector space $X$, we denote by $D_+'(\R,X)$ (resp., $\cS_+'(\R,X)$) the set of $X$-valued distributions (resp., tempered distributions) in $D'(\R,X)$ (resp., $\cS'(\R,X)$) that are supported in $[0,+\infty)$. Evidently, any distribution in $D_+'(\R,X')$, where $X'$ is the topological space dual to $X$, may be regarded as a continuous linear form in $D_{\mathrm a}(\R,X):=\{\psi \in C^\infty(\R,X);\ \exists R>0,\ \supp \psi \subset(-\infty,R)\}$, 
%of smooth $L^2(\Omega)$-valued functions whose support is bounded from above
endowed with the associated canonical LF-topology. We denote by $\langle \cdot , \cdot \rangle_{{\mathrm a},X}$ or simply by $\langle \cdot , \cdot \rangle$ when there is no ambiguity, the corresponding duality pairing. Thus, bearing in mind that $s\mapsto\int_0^{+\infty}K(t,\cdot)\psi(t+s,\cdot)dt\in D_{\mathrm a}(\R,L^2(\Omega))$ whenever $\psi\in D_{\mathrm a}(\R,L^2(\Omega))$, we extend $I_K$ as a continuous linear map from $D_+'(\R,L^2(\Omega))$ to $D_+'(\R,L^2(\Omega))$ by setting
$$\langle I_Kv,\psi \rangle := \left \langle v(s,\cdot),\int_0^{+\infty}K(t,\cdot)\psi(t+s,\cdot)dt \right \rangle,\ v \in D_+'(\R,L^2(\Omega)),\ \psi\in D_{\mathrm a}(\R,L^2(\Omega)), $$
where, as usual, $L^2(\Omega)$ is identified with its dual space.
Next, with reference to the definition \cite[Chap. 6, Section 5, Formula 15]{S} of non-integer order distributional derivatives, we introduce the Riemann-Liouville (resp. Caputo) fractional derivative $D_{t,K}$ (resp., $\partial_{t,K}$) with kernel $K$ as 
$D_{t,K}v:=\partial_t I_Kv$ (resp., $\partial_{t,K} v := I_K \partial_t v$) for all $v\in D_+'(\R,L^2(\Omega))$.

\subsection{Initial boundary value problem with a singular source term}

We consider the initial boundary value problem (IBVP) with initial state $u_0$ and source $F$,
\bel{eq1}
\left\{ \begin{array}{rcll} 
( \partial_{t,K}+\cA ) u(t,x) & = & F(t,x), & (t,x)\in 
\R_+\times\Omega\\
 u(t,x) & = & 0, & (t,x) \in \R_+\times\partial\Omega \\  
 u(0,x) & = & u_0(x), & x \in \Omega.
\end{array}
\right.
\ee
Here and below, we set
$$ 
\cA u(x) :=-\sum_{i,j=1}^d \partial_{x_i} 
\left( a_{i,j}(x) \partial_{x_j} u(x) \right)+q(x)u(x),\ x \in \Omega,
$$ 
where $q \in L^{\frac{d}{2}}(\Omega)$ is non-negative and 
$a:=(a_{i,j})_{1 \leq i,j \leq d} \in L^\infty(\Omega,\R^{d^2})\cap H^1(\Omega,\R^{d^2})$
is symmetric and satisfies the following ellipticity condition:
\bel{ell}
\exists c>0,\ \sum_{i,j=1}^d a_{i,j}(x) \xi_i \xi_j \geq c |\xi|^2,\ x \in \Omega,\ \xi=(\xi_1,\ldots,\xi_d) \in \R^d.
\ee
%for some positive constant $c$. 
Let $u_0\in L^2(\Omega)$ and $F\in D_+'(\R,L^2(\Omega))$. 
\begin{defn}
\label{d1}
We say that $u\in \cS_+'(\R,L^2(\Omega))$ is a weak-solution to \eqref{eq1} if the two following conditions are satisfied simultaneously:
\begin{enumerate}[{\rm(i)}]
\item 
$
\left \langle D_{t,K} u+\mathcal A u,\psi\right\rangle_{{\mathrm a},D(\Omega)}=\left\langle K_+  u_0 + F,\psi \right\rangle_{{\mathrm a},L^2(\Omega)}$ for all $\psi\in D_{\mathrm a}(\R,D(\Omega))$,
where $K_+(t,x):=K(t,x) \mathds{1}_{\R_+}(t)$ and $\mathds{1}_{\R_+}$ is the characteristic function of $\R_+$.
\item For all $p\in\mathbb C_+:=\{z\in\mathbb C;\ \re z>0\}$, the Laplace transform (with respect to $t$) of $u$ at $p$, defined for a.e. $x \in \Omega$ by $\wh u(p,x): = \left\langle u(t,x),e^{-pt}\right\rangle_{\cS_+'(\R),\cS_+(\R)}$ where $\cS_+(\R):=\{ \phi\in C^\infty(\R);\ \phi_{|\R_+}\in \cS(\R_+)\}$, lies in $H_0^1(\Omega)$.
\end{enumerate}
\end{defn}

\subsection{Anomalous diffusion processes}
\label{ref-adp}
The main purpose of this article is to examine the existence and the uniqueness issues of a weak solution to \eqref{eq1} in the presence of a singular source term $F$, for each of the three following classical types of diffusion models:\begin{enumerate}[1)]
\item \textit{Space-dependent variable order diffusion equations} associated with $\alpha\in L^\infty(\Omega)$ satisfying
$0 < \alpha_0 \leq \alpha(x) \leq \alpha_M<1$ for a.e. $x \in \Omega$, where $\alpha_M<2\alpha_0$, and
\bel{Kvariable} 
K(t,x):=\frac{t^{-\alpha(x)}}{\Gamma(1-\alpha(x))},\ t\in\R_+,\ x\in\Omega.
\ee
\item \textit{Distributed order time-fractional diffusion } associated with a non-negative weight function $\mu \in L^\infty(0,1)$ such that $\exists \alpha_0 \in(0,1),\ \exists \epsilon \in (0,\alpha_0),\ \forall \alpha \in (\alpha_0-\epsilon,\alpha_0),\ \mu(\alpha) \ge \frac{\mu(\alpha_0)}{2}>0$, and
\bel{Kdistributed} 
K(t,x):=\int_0^1 \mu(\alpha)\frac{t^{-\alpha}}{\Gamma(1-\alpha)}d\alpha,\quad t\in\R_+,\ x\in\Omega.
\ee
\item \textit{Multiterm time-fractional diffusion models} associated with $1<\alpha_1<\ldots<\alpha_N<1$ for some $N\in\mathbb N:=\{1,2,\ldots\}$, $\rho_j\in L^\infty(\Omega)$ such that $0<c_0 \leq\rho_j(x) \leq C_0 <+\infty$ for $j=1,\ldots,N$ and a.e. $x \in \Omega$, and 
\bel{Kmultiple} 
K(t,x):=\sum_{j=1}^N\rho_j(x) \frac{t^{-\alpha_j}}{\Gamma(1-\alpha_j)},\ t\in\R_+,\ x\in\Omega.
\ee
\end{enumerate}
Anomalous diffusion in a heterogeneous medium is a growing issue of scientific research, with numerous applications  areas such as geophysics, hydrology or biology, see e.g. \cite{CZZ,FS,FH}. Some typical examples are fluid flow in porous media, propagation of seismic waves, and protein dynamics. In this situation the variations of permeability in different spatial positions caused by the heterogeneities of the medium induce location dependent diffusion phenomena 
which are correctly described by the space dependent model \eqref{eq1} associated with the kernel \eqref{Kvariable}. On the other hand, \eqref{Kdistributed} is used  for modeling ultra slow diffusion processes whose mean square displacement scales like a log with respect to the time variable. For instance, such  phenomena were observed in polymer physics or the kinetics of particles moving in quenched random force fields, see e.g. 
in \cite{N,SCK}. Finally the multi-term time-fractional diffusion equation \eqref{eq1} associated with \eqref{Kmultiple} is a useful tool for modeling the behavior of viscoelastic fluids and rheological material, see e.g. \cite{FLTZ}.

\section{Singular sources and the well-posedness issue}

\subsection{What we are aiming for}
The well-posedness issue of constant-order fractional diffusion equations (obtained from \eqref{eq1} by setting $K(t,x):=\frac{t^{-\beta}}{\Gamma(1-\beta)}$ for some fixed $\beta\in(0,1)$)  has received a great deal of attention from the mathematical community over the last decades, see e.g. \cite{EK,SY,KY2} and the references therein. Similarly, several techniques were used in \cite{KSY,LHY,LKS,LLY2,LKS} to build a solution to variable-order, distributed or multi-term time-fractional processes, and we refer the reader to \cite{K2} for the study of the equivalence of these different approaches. All the above mentioned works assume that the source term $F$ is within the class $L^1_{\mathrm{loc}}(\R_+,L^2(\Omega))$. Recently, the well-posedness of constant-order time-fractional diffusion systems was examined in \cite{Y}, under the assumption that $t \mapsto F(t,\cdot)$ lies in a negative order Sobolev space. The aim of this article is to extend the study carried out in \cite{Y} in two main directions: Firstly, by adapting the analysis to the wider class of diffusion equations described in Section \ref{ref-adp} and, secondly, by considering source terms $F$ with a distributional time-dependence only.

\subsection{Statement of the result}

Let $\cR\in L^1_{\mathrm{loc}}(\R,B(L^2(\Omega)))\cap D_+'(\R,\cB(L^2(\Omega)))$.
%where $\cB(L^2(\Omega))$ denotes the set of linear bounded operators in $L^2(\Omega)$. 
Since $s\mapsto \int_\R \cR(t)^*\phi(t+s,\cdot)dt=\int_0^{+\infty} \cR(t)^*\phi(t+s,\cdot)dt\in D_{\mathrm a}(\R,L^2(\Omega))$ whenever $\phi\in D_{\mathrm a}(\R, L^2(\Omega))$, where $R^*(t)$ is the adjoint operator to $R(t)$, we define $\cR * v$ for any $v\in D_+'(\R,L^2(\Omega))$, as the unique distribution in $D_+'(\R,L^2(\Omega))$ such that
$$ \left\langle \cR * v ,\phi \right\rangle_{{\mathrm a},L^2(\Omega)}=\left\langle v(s,\cdot),\int_\R \cR(t)^*\phi(t+s,\cdot)dt\right\rangle_{{\mathrm a},L^2(\Omega)},\ \phi\in D_{\mathrm a}(\R, L^2(\Omega)). $$
As a consequence we have $\partial_t^m(\cR* v)=\cR* (\partial_t^mv)$ for all $m\in\mathbb N$.

For $\theta\in(\frac{\pi}{2},\pi)$ and $\delta \in (0,+\infty)$ fixed, put 
\begin{equation}\label{g2}
\gamma_0(\delta,\theta):=\{\delta\, e^{i\beta}:\ \beta\in[-\theta,\theta]\},\quad\gamma_\pm(\delta,\theta)
:=\{s\,e^{\pm i\theta}\mid s\in[\delta,\infty)\},
\end{equation}
and let the contour $\gamma(\delta,\theta):=\gamma_-(\delta,\theta)\cup\gamma_0(\delta,\theta)\cup\gamma_+(\delta,\theta)$ be oriented in the counterclockwise direction in $\mathbb C$.
Next, $K$ being either of the three expressions \eqref{Kvariable}, \eqref{Kdistributed} or \eqref{Kmultiple}, we refer to \cite{K2,KSY,LKS} and set for all $\psi \in L^2(\Omega)$, 
\begin{equation} 
\label{S0}
S_{j,K}(t)\psi:=\left\{\begin{array}{ll} \frac{1}{2i\pi}\int_{\gamma(\delta,\theta)}e^{t p}\left(A+p\wh{K}(p,\cdot)\right)^{-1} \wh{K}(p,\cdot)^{1-j}\psi d p & \textrm{if}\ t>0\\
0 & \textrm{if}\ t\leq0
\end{array} \right., j=0,1.
\end{equation}
Here and in the remaining part of this text, $A$ denotes the self-adjoint operator in $L^2(\Omega)$ acting as $\mathcal A$ on its domain $D(A):=\{h\in H^1_0(\Omega);\ \mathcal A h\in L^2(\Omega)\}$. 
We have $S_{j,K}\in L^1_{\mathrm{loc}}(\R,L^2(\Omega))\cap\mathcal S_+'(\R,\cB(L^2(\Omega)))$, $j=0,1$, by \cite[Lemma 6.1]{KLY} and \cite[Lemmas 2.1, 3.1 \& 4.2]{K2} and it is apparent that
\begin{equation}
\label{sol}
u(t,\cdot):=S_{0,K}(t)u_0+S_{1,K}*F(t,\cdot),\ t\in\R,
\end{equation}
lies in $D_+'(\R,L^2(\Omega))$ whenever $F\in D_+'(\R,L^2(\Omega))$. Moreover, \eqref{sol} reads
$$u(t,\cdot)=S_{0,K}(t)u_0+\int_0^tS_{1,K}(t-s)F(s,\cdot)ds,\ t\in\R, $$
provided that $F\in L^1_{\mathrm{loc}}(\R,L^2(\Omega))\cap D_+'(\R,L^2(\Omega))$ and if we assume in addition that 
$(1+t)^{-N}F\in L^1(\R_+;L^2(\Omega))$ for some $N\in\mathbb N$, then it can be checked from \cite[Theorems 1.3, 1.4 \& 1.5]{K2}
that \eqref{sol} is a weak solution to \eqref{eq1} in the sense of Definition \ref{d1}. The main achievement of this short article is the following extension of the above claim to the case of a source term $F\in \cS_+'(\R,L^2(\Omega))$.

\begin{thm}
\label{t1}
Let $K$ be given by either of the three expressions \eqref{Kvariable}, \eqref{Kdistributed} or \eqref{Kmultiple}. Assume \eqref{ell} and pick $u_0\in L^2(\Omega)$ and $F\in \cS_+'(\R,L^2(\Omega))$. Then, \eqref{sol} is the unique weak solution to \eqref{eq1} in the sense of Definition \ref{d1}.
\end{thm}
The proof of Theorem \ref{t1} is given in Section \ref{sec-pr}. 

\subsection{Brief comments}
Notice that the statement of Theorem \ref{t1} includes a unique weak solution enjoying the Duhamel representation formula \eqref{sol},
to anomalous diffusion processes governed by \eqref{eq1} with a delta-singular source term of the form
$$ F(t,x)=\sum_{j=1}^N \delta_{t_j}^{(k_j)}(t)f_j(x),\ t\in\R,\ x\in\Omega,$$
where $0 \leq t_1<t_2<\ldots<t_N < +\infty$, $k_1,\ldots,k_N \in \mathbb N\cup\{0\}$ and $f_1,\ldots,f_N\in L^2(\Omega)$ for some $N \in \N$.

Definition \ref{d1} of a weak solution to diffusion equations with a variable, distributed or multi-term fractional derivative and a possibly highly singular source term $F$ in $D_+'(\R,L^2(\Omega))$, generalizes the ones given in \cite{EK,K2,KSY,KJ,LKS} and in the references therein, in the context of more specific diffusion processes. Moreover, the representation formula of the weak solution given in \cite{K2,KSY,LKS} can be deduced from
the Duhamel formula \eqref{sol} provided by Theorem \ref{t1} under the more general assumption that $F\in \cS_+'(\R,L^2(\Omega))$. %Theorem \ref{t1} not only generalizes the analysis carried out in \cite{EK,K2,KSY,KJ,LKS}, but pushes the analysis 

\section{Proof of Theorem \ref{t1}}
\label{sec-pr}
%This section is devoted to the proof of the main theorem. We start by considering the following intermediate result.
Prior to showing Theorem \ref{t1}, we establish a technical result needed by the proof.
\subsection{Preliminaries}
The result is as follows.
\begin{lem}
\label{l1} 
Let $K$ be given by either of the formulas \eqref{Kvariable}, \eqref{Kdistributed} or \eqref{Kmultiple}. Then, for all $u \in \cS_+'(\R,L^2(\Omega))$ we have $I_Ku\in \cS_+'(\R,L^2(\Omega))$ and the identity
\begin{equation}
\label{l1a}
\wh{I_Ku}(p,\cdot)=\wh{K}(p,\cdot)\wh{u}(p,\cdot),\ p\in\mathbb C_+, 
\end{equation}
where we recall that $\wh u$ denotes the Laplace transform of $u$.
\end{lem}
\begin{proof}
For all $\varphi \in D_{\mathrm a}(\R,L^2(\Omega))$, we have 
$\left\langle I_Ku,\phi\right\rangle_{{\mathrm a},L^2(\Omega)}=\left\langle u,\phi_K\right\rangle_{\cS_+'(\R,L^2(\Omega)),\cS_+(\R,L^2(\Omega))}$, 
where
$\phi_K(t,\cdot):=\int_0^{+\infty}K(s,\cdot)\phi(s+t,\cdot)$.
Since $\supp u \subset[0,+\infty)$, it follows from this that there exists two natural numbers $m_1$ and $m_2$ such that
\begin{equation}
\label{eq-pr1}
\abs{\left\langle I_Ku,\phi\right\rangle_{{\mathrm a}, L^2(\Omega)}}\leq C\sup_{t\in[0,+\infty)}\norm{(1+t)^{m_2}\partial_t^{m_1}\phi_K(t)}_{L^2(\Omega)}.
\end{equation}
Here and in the remaining part of this proof, $C$ denotes a generic positive constant which may change from line to line.
Furthermore, taking into account that $K\in L^1(0,1;L^\infty(\Omega))\cap L^\infty(1,+\infty;L^\infty(\Omega))$ whenever $K$ is given by \eqref{Kvariable}, \eqref{Kdistributed} or \eqref{Kmultiple}, we obtain for all $t\in[0,+\infty)$ that
\begin{eqnarray*}
&&\norm{(1+t)^{m_2}\partial_t^{m_1}\phi_K(t)}_{L^2(\Omega)}\\
&\leq & \int_0^{+\infty}\norm{K(s,\cdot)}_{L^\infty(\Omega)}\norm{(1+t+s)^{m_2}\partial_t^{m_1}\phi(s+t,\cdot)}_{L^2(\Omega)}ds\\
& \leq &\left(\norm{K}_{L^1(0,1;L^\infty(\Omega))}+\norm{K}_{L^\infty(1,+\infty;L^\infty(\Omega))}\int_1^{+\infty}(1+t+s)^{-2}ds\right)\sup_{t\in[0,+\infty)}\norm{(1+t)^{m_2+2}\partial_t^{m_1}\phi(t)}_{L^2(\Omega)}\\
&\leq & C\sup_{t\in[0,+\infty)}\norm{(1+t)^{m_2+2}\partial_t^{m_1}\phi(t)}_{L^2(\Omega)}.
\end{eqnarray*}
From this and \eqref{eq-pr1} it then follows that
$\abs{\left\langle I_Ku,\phi\right\rangle_{{\mathrm a},L^2(\Omega)}}\leq C\sup_{t\in[0,+\infty)}\norm{(1+t)^{m_2+2}\partial_t^{m_1}\phi(t)}_{L^2(\Omega)}$. Thus, $I_Ku$ extends by density to a vector of $\cS_+'(\R,L^2(\Omega))$. 
Moreover, for all $p\in\mathbb C_+$, it holds true that
\begin{eqnarray*}
\wh{I_Ku}(p,\cdot)&=&\left\langle u(t,\cdot),\int_0^{+\infty}K(s,\cdot)e^{-p(t+s)}ds\right\rangle_{\mathcal S'_+(\R),\mathcal S_+(\R)}\\
&=&\left\langle u(t,\cdot),e^{-pt}\int_0^{+\infty}K(s,\cdot)e^{-ps}ds\right\rangle_{\mathcal S'_+(\R),\mathcal S_+(\R)}\\
&=& \left(\int_0^{+\infty}K(s,\cdot)e^{-ps}ds\right)\left\langle u(t,\cdot),e^{-pt}\right\rangle_{\mathcal S'_+(\R),\mathcal S_+(\R)},\end{eqnarray*}
which yields \eqref{l1a}.\end{proof}

Armed with Lemma \ref{l1} we turn now to showing Theorem \ref{t1}.

\subsection{Completion of the proof}
With reference to \cite[Theorems 1.3, 1.4 \& 1.5]{K2}, we assume without loss of generality that $u_0=0$ in $\Omega$. We examine the uniqueness and the existence issues separately. We start with the uniqueness.\\
\noindent\textit{Uniqueness}. Let $u$ be a weak solution (in the sense of Definition \ref{d1}) to \eqref{eq1} associated with $F=0$ in $\mathbb R_+\times\Omega$ and $u_0=0$ in $\Omega$. Then, we have $D_{t,K}u=\partial_tI_Ku\in \cS_+'(\R,L^2(\Omega))$ by Lemma \ref{l1} and $\mathcal Au\in \mathcal S'_+(\R;D'(\Omega))$. Moreover, $\wh{D_{t,K}u}(p,\cdot)=p\wh{I_Ku}(p,\cdot)=p\wh{K}(p,\cdot)\wh{u}(p,\cdot)$ and $\wh{\mathcal Au}(p,\cdot)=\mathcal A\wh{u}(p,\cdot)$ for all $p\in\mathbb C_+$.
Therefore, applying the Laplace transform with respect to $t$ to both sides of \eqref{eq1}, we get that
$$p\wh{K}(p,x)\wh{u}(p,x)+\mathcal A\wh{u}(p,x) %=\wh{D_{t,K}u+\mathcal Au}(p,x)=0
,\ p\in\mathbb C_+,\ x\in\Omega.$$
From this and the fact that $\wh{u}(p,\cdot)\in H^1_0(\Omega)$ for all $p\in\mathbb C_+$, according to Definition \ref{d1}(ii), we get that $\wh{u}(p,\cdot)=0$ in $\Omega$ by \cite[Proposition 2.1]{KSY} and \cite[Lemma 4.1]{K2}. Thus, $u=0$ in $\mathbb R_+\times\Omega$ from the injectivity of the Laplace transform. This proves that a weak solution to \eqref{eq1}, if any, is unique.\\
\noindent \textit{Existence}. Let us establish that $u\in D_+'(\R,L^2(\Omega))$, expressed by \eqref{sol}, is a weak solution to \eqref{eq1}. Since $F\in \mathcal S'(\R,L^2(\Omega))$, then there exist $F_1\in L^1_{\mathrm{loc}}(\R,L^2(\Omega))$ and $N_1 \in \N$ such that
$F=\partial_t^{N_1}F_1$, according to \cite[Theorem 8.3.1]{FJ}. Moreover we have $(1+t^2)^{-N_2}F_1\in L^\infty(\R,L^2(\Omega))$ for some $N_2 \in \N$. 
%From this and the fact that $F\in \cS_+'(\R,L^2(\Omega))$, it then follows that $F_1(t,\cdot)=0$ for all $t<0$. 
Thus, applying \eqref{sol} with $u_0=0$, we get that
\begin{equation}
\label{eq-pr2}
u=S_{1,K}*F=S_{1,K}*(\partial_t^{N_1}F_1)=\partial_t^{N_1}(S_{1,K}*F_1),
\end{equation}
where $S_{1,K}*F_1(t)=\int_0^tS_{1,K}(t-s)F_1(s)ds$ for all $t\in \R$, as we have $F_1\in L^1_{\mathrm{loc}}(\R,L^2(\Omega))$. The next step is to notice from \cite[Lemmas 2.1, 3.1 \& 4.2]{K2} that there exists $r_1 \in (0,1)$ such that
$\norm{S_{1,K}(t)}_{\cB(L^2(\Omega))}\leq C\max(1,t^{-r_1})$
for some positive constant $C$, whenever $t \in \R_+$ and $K$ is given by either \eqref{Kvariable}, \eqref{Kdistributed} or \eqref{Kmultiple}. From this
and the fact that $(1+t^2)^{-N_2}F_1\in L^\infty(\R,L^2(\Omega))$, it then follows that $S_{1,K}*F_1\in\cS_+'(\R,L^2(\Omega))$. 
Therefore, we have $u\in \cS_+'(\R,L^2(\Omega))$ from \eqref{eq-pr2}.

Furthermore, since $(1+t)^{-2N_2-2}F_1\in L^1(\R_+,L^2(\Omega))$, we infer from \cite[Propositions 2.2, 3.2 \& 4.5]{K2} that 
\begin{equation}
\label{eq-pr3}
\cL \left( S_{1,K}*F_1 \right)(p,\cdot)=\left(A+p\wh{K}(p,\cdot)\right)^{-1}\wh{F_1}(p,\cdot),\ p\in\mathbb C_+,
\end{equation}
where $\cL$ stands for the Laplace transformation with respect to $t$. In view of \eqref{eq-pr2}, giving $\wh{u}(p,\cdot) =\cL \left( \partial_t^{N_1}S_{1,K}*F_1 \right)(p,\cdot)  = p^{N_1}\cL \left( S_{1,K}*F_1 \right)(p,\cdot)$ for all $p\in\mathbb C_+$, we deduce from \eqref{eq-pr3} that
\begin{eqnarray}
\wh{u}(p,\cdot) %& =& \cL \left( \partial_t^{N_1}S_{1,K}*F_1 \right)(p,\cdot) \nonumber \\
%& = & p^{N_1}\cL \left( S_{1,K}*F_1 \right)(p,\cdot) \nonumber\\
= \left(A+p\wh{K}(p,\cdot)\right)^{-1}p^{N_1}\wh{F_1}(p,\cdot) &=& \left(A+p\wh{K}(p,\cdot)\right)^{-1}\cL \left( \partial_t^{N_1}F_1 \right)(p,\cdot) \nonumber \\
& = & \left(A+p\wh{K}(p,\cdot)\right)^{-1}\wh{F}(p,\cdot). \label{t1c}
\end{eqnarray}
Since $\wh{F}(p,\cdot)\in L^2(\Omega)$ for all $p\in\mathbb C_+$, then we have $\wh{u}(p,\cdot)\in D(A)\subset H^1_0(\Omega)$, showing that Definition \ref{d1}(ii) is satisfied.
Finally, putting \eqref{t1c} together with Lemma \ref{l1}, we get that $v:= D_t^Ku+\mathcal A u$ lies in $\mathcal S'_+(\R,D'(\Omega))$ and that 
$$\wh{v}(p,\cdot)=(\mathcal A+p\wh{K}(p,\cdot))\wh{u}(p,\cdot)=(A+p\wh{K}(p,\cdot))\left(A+p\wh{K}(p,\cdot)\right)^{-1}\wh{F}(p,\cdot)=\wh{F}(p,\cdot),\ p\in\mathbb C_+. 
$$
As a consequence we have $D_{t,K}u+\mathcal A u=F$ in $\mathcal S'_+(\R,D'(\Omega))$, by injectivity of the Laplace transform $\cL$, which entails Definition \ref{d1}(i). This proves that $u$ is a weak solution to \eqref{eq1}.

\section*{Acknowledgments}
 This work was supported by the Agence Nationale de la Recherche (ANR) under grant ANR-17-CE40-0029 (projet MultiOnde).

%%%%%%%%%%%%%%%%%%%%%%%%%%%%%%%%%%%%%%%%%%%%%%%%%%%%%%%%%%%%%%%%%
%%%%%%%%%%%%%%%%%%%%%%%%%%%%%%%%%%%%%%%%%%%%%%%%%%%%%%%%%%%%%%%%%
%%%%%%%%%%%%%%%%%                                  Bibliography                             %%%%%%%%%%%%%%%%%%%%%%
%%%%%%%%%%%%%%%%%%%%%%%%%%%%%%%%%%%%%%%%%%%%%%%%%%%%%%%%%%%%%%%%%
%%%%%%%%%%%%%%%%%%%%%%%%%%%%%%%%%%%%%%%%%%%%%%%%%%%%%%%%%%%%%%%%%
 
\end{document}